\documentclass[12pt]{gtpart}
\usepackage{geometry,microtype}                
\usepackage{graphicx}
\usepackage{amssymb}
\usepackage{epstopdf}
\newtheorem{thm}{Theorem}[section]

\newtheorem{cor}[thm]{Corollary}
\newtheorem{prop}[thm]{Proposition}
\newtheorem{lem}[thm]{Lemma}

\theoremstyle{definition}

\newtheorem{notation}[thm]{Notation}

\theoremstyle{remark}

\def\vol[#1]{\mbox{\rm Vol}(#1)}

\def\lv[#1]{\mbox{\rm LinkVol}(#1)}

\def\vvv{V_{0}}

\def\wt{\widetilde}

\def\f21{F_{2,1}}

\title{The Link Volumes of some prism manifolds}
\author{Jair Remigio}
\givenname{Jair}
\surname{Remigio-Ju\'arez}
\address{Divisi\'on Acad\'emica de Ciencias B\'asicas\\Universidad Ju\'arez Aut\'onoma de Tabasco,\\\newline 
 Km. 1 Carr. Cunduac\'an-Jalpa de M\'endez\\Cunduac\'an, Tab. 86690\\Mexico}
\email{jair.remigio@ujat.mx}
\urladdr{}

\author{Yo'av Rieck}
\givenname{Yo'av}
\surname{Rieck}
\address{Department of Mathematical Sciences\\University of Arkansas,\\\newline
Fayetteville, AR, 72701\\USA}
\email{yoav@uark.edu}
\date{\today}                                           

\volumenumber{}
\issuenumber{}
\publicationyear{}
\papernumber{}
\startpage{}
\endpage{}
\doi{}
\MR{}
\Zbl{}
\received{}
\revised{}
\accepted{}
\published{}
\publishedonline{}
\proposed{}
\seconded{}
\corresponding{}
\editor{}
\version{}

\begin{document}

\begin{abstract}
We calculate the link volume of an infinite family of prism manifolds.  As a corollary, we show that the
link volume is not finite-to-one.
\end{abstract}

\maketitle

\section{Introduction}
\label{sec:intro}

In~\cite{ry} the second named author and Y. Yamashita defined the {\it link volume}, an invariant of closed orientable 3-manifolds
that measures how efficiently a given manifold can be represented as a branched cover of $S^{3}$.
We use the notation $M \stackrel{p}{\to} (S^3,L)$ to denote a covering projection from $M$ to $S^{3}$, branched along $L$
and of degree $p$.  We restrict to the case where $L$ is a hyperbolic link.  Then the {\it complexity}
of $M \stackrel{p}{\to} (S^3,L)$ is defined to be $p \vol[S^{3} \setminus L]$, that is, the degree of the cover
times the volume of the complement of the branch set.
The link volume of a closed orientable 3-manifold $M$ is denoted $\lv[M]$ and defined to be the infimum of the complexities of all
covers (of all possible degrees) $M \stackrel{p}{\to} (S^3,L)$,  that is:
$$\lv[M] = \inf \{ p \vol[S^3 \setminus L] | M \stackrel{p}{\to} (S^3,L); L \mbox{ hyperbolic}\}.$$
In~\cite{ry} the basic properties of link volume are explored, and it is shown that many of these properties are similar to the
corresponding properties of the hyperbolic volume.  However, the link volume seems quite challenging to calculate in general.  

In this paper we calculate the link volume of all but finitely many members
of an infinite family of prism manifolds denoted $M_{n}$.  $M_{n}$ is defined to be the 2-fold cover of $S^{3}$ branched
along the link $L_{n}$, see Figure~\ref{fig:LnSimple}.  

\begin{figure}[h!]
\centering
	\includegraphics[height = 1in]{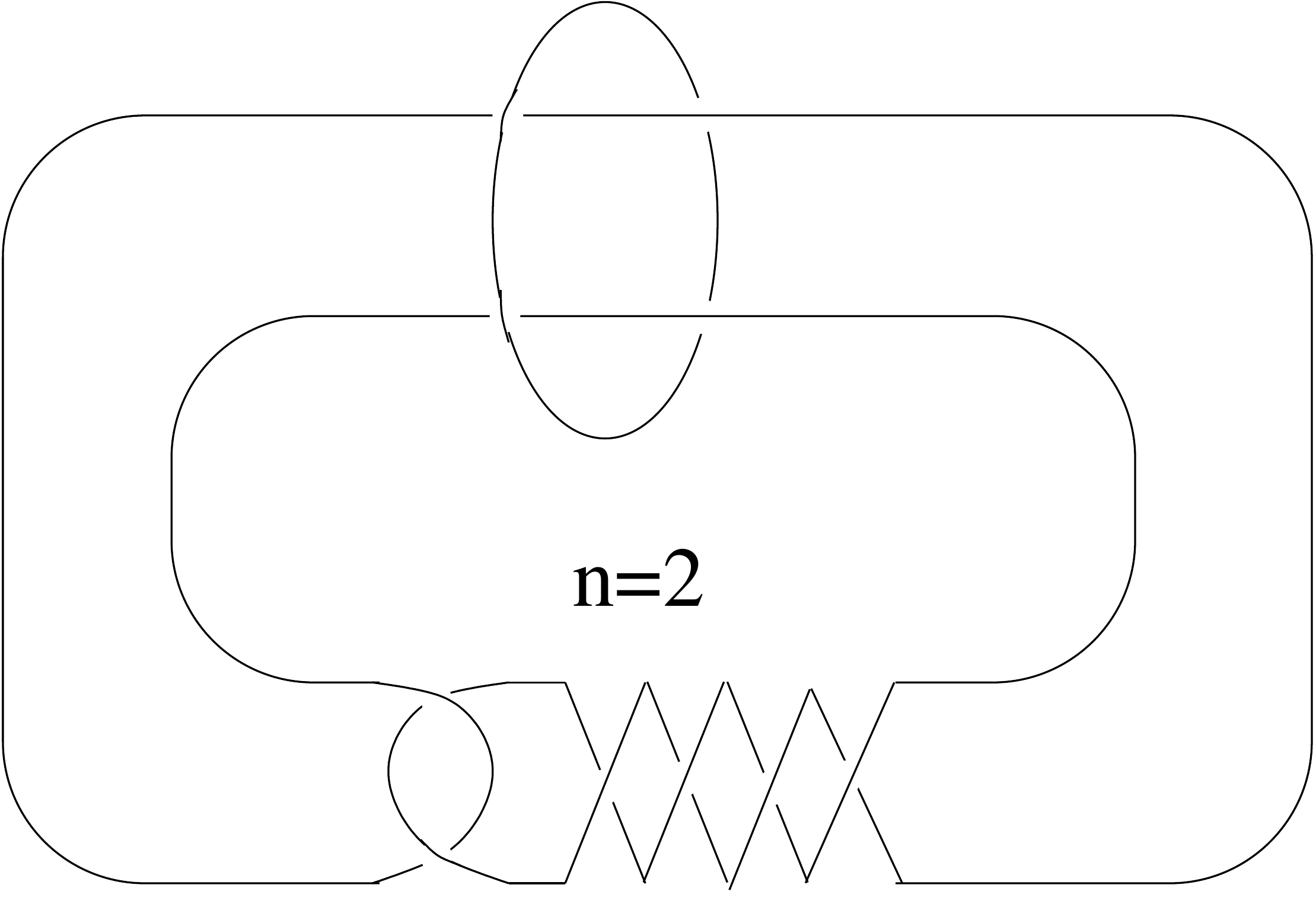}
	\caption{$L_n$}
	\label{fig:LnSimple}
	\end{figure}

The manifolds $M_{n}$ are all 
Seifert manifolds with finite fundamental group.  Each $M_{n}$ admits exactly two Seifert fibrations, one with Seifert symbols $(On,1;\frac{4n-1}{2})$
and the other with Seifert symbols $(Oo,0;\frac{1}{2}, -\frac{1}{2}, \frac{-2}{4n-1})$.  For details and other background material see Section~\ref{sec:background}.
Let $\vvv(=3.66\ldots)$ denote the volume of the Whitehead link exterior.  Note that the Whitehead link is $L_{0}$, and since 
$S^{3} \setminus L_{n} \cong S^{3} \setminus L_{0}$, $\vol[S^{3} \setminus L_{n}] = \vvv$ for all $n$.
In~\cite{agol} I. Agol proved that  $S^{3} \setminus L_{0}$ and $S^{3} \setminus P$ are the 2-cusped hyperbolic
manifolds of smallest volume, where $P$ is the $-2,3,8$ pretzel link.  Using this result we obtain:

\begin{thm}
\label{thm:main}
For all but finitely many $n$, $\lv[M_{n}] = 2\vvv$.
\end{thm}

It is well known that the  hyperbolic volume function is finite to one.  In~\cite{ry} it was asked whether the link volume
is finite to one as well.  By Lemma~\ref{lem:Mn} the manifolds $M_{n}$ are
all distinct; combining these results we obtain:

\begin{cor}
There exist infinitely many distinct manifolds with the same link volume.
\end{cor}

After going over background material in Section~\ref{sec:background}, in Section~\ref{sec:Pn} we prove
(Proposition~\ref{prop:Pn}) that the knots that are obtained from the $-2,3,8$ pretzel link
via Dehn surgery on the unknotted component
are all {\it twisted torus knots} (see Subsection~\ref{subsection:notation} for definition)
of the form  $T(5,5n+1;2,1)$.
In Section~~\ref{sec:SmallVolumeKnots} we prove Proposition~\ref{prop:SmallVolumeKnots},
which is of independent interest.  In it we show
that all but finitely many of the knots that have volume less than $\vvv$ are either twist knots
or $T(5,5n+1;2,1)$ twisted torus knots.
Finally, in Section~\ref{sec:proof}, we prove Theorem~\ref{thm:main}.

\bigskip

\noindent{\bf Acknowledgments.}  We thank Kimihiko Motegi and Yasushi Yamashita for helpful 
conversations.  The first named author was supported by CONACYT Postdoctoral Fellowships Program and thanks the Department of Mathematical Sciences-University of Arkansas for the facilities provided to carry out this project.

\section{Background}
\label{sec:background}

By {\it manifold} we mean 3-dimensional manifold.
Unless otherwise stated, all manifolds and surfaces considered are assumed to be connected and orientable.
In addition, every manifold is assumed to be compact or obtained from a compact manifold
by removing some of its boundary components.  
A {\it knot} is a smooth embedding of $S^{1}$ into $S^{3}$.  A {\it link} is collection of disjoint knots.

A closed orientable surface of positive genus embedded in a 3-manifold is called {\it incompressible}
if the inclusion map induces a monomorphism between the fundamental groups.  An incompressible
surface is called {\it essential} if it is not boundary parallel.  A manifold is called
{\it a-toroidal} if it does not admit an essential torus.  A surface with non-empty boundary
properly embedded in a manifold is called {\it essential} if it is incompressible, boundary
incompressible and not boundary parallel.  For details about this and other standard notions in 3-manifold
topology see, for example,~\cite{hempel} or~\cite{jaco}.

\subsection{Notation}
\label{subsection:notation}

Throughout this paper, we use the following notation: the Whitehead link is denoted $W$, 
see Figure~\ref{fig:whitehead}.  The 2-component
link obtained by $n$ Dehn twist on one of the components of $W$ 
is denoted $L_{n}$.  The knot obtained 
from one component of the $W$ by $1/n$ Dehn surgery about the other component is called a {\it twist
knot}, denoted $W_{n}$, see Figure~\ref{fig:TwistKnot}.  The $-2,3,8$ Pretzel link is denoted $P$, see Figure~\ref{fig:PretzelLink}.  
Note that one component of $P$ is unknotted
and the other is a trefoil.  The knot obtained by $1/n$ Dehn surgery on the unknotted component of $P$ is
denoted by $P_{n}$.  The $p/q$ torus knot is denoted $T(p,q)$.  The {\it twisted torus knot} $T(p,q;r,s)$ is the
knot obtained from $T(p,q)$ by performing $s$ full Dehn twists on $r$ strands; see Figure~\ref{fig:ttk}.
	\begin{figure}[h!]
\centering
	\includegraphics[height = 1in]{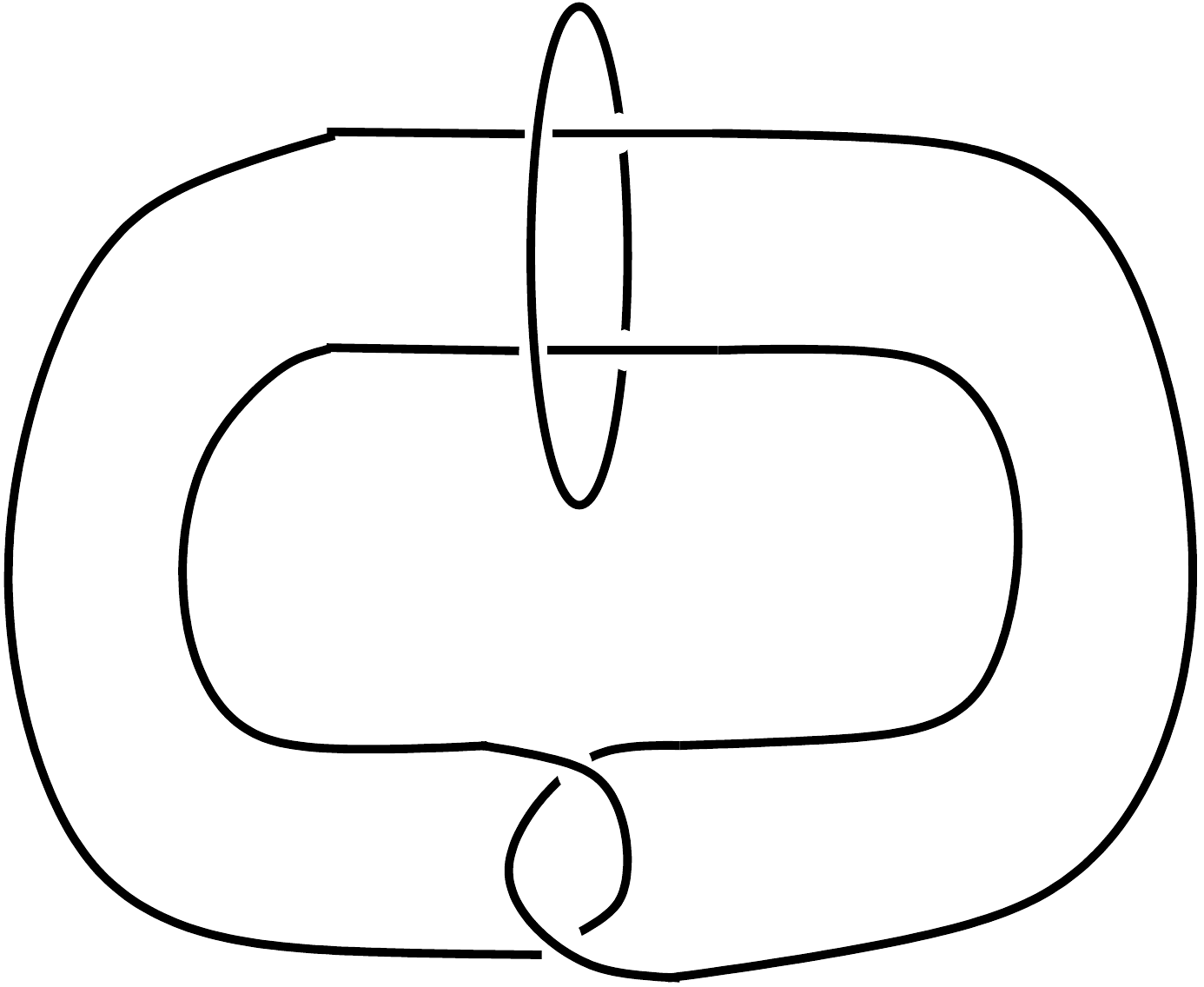}
	\caption{$W$}
	\label{fig:whitehead}
	\end{figure}

	\begin{figure}[!h]
\centering
	\includegraphics[height = 1in]{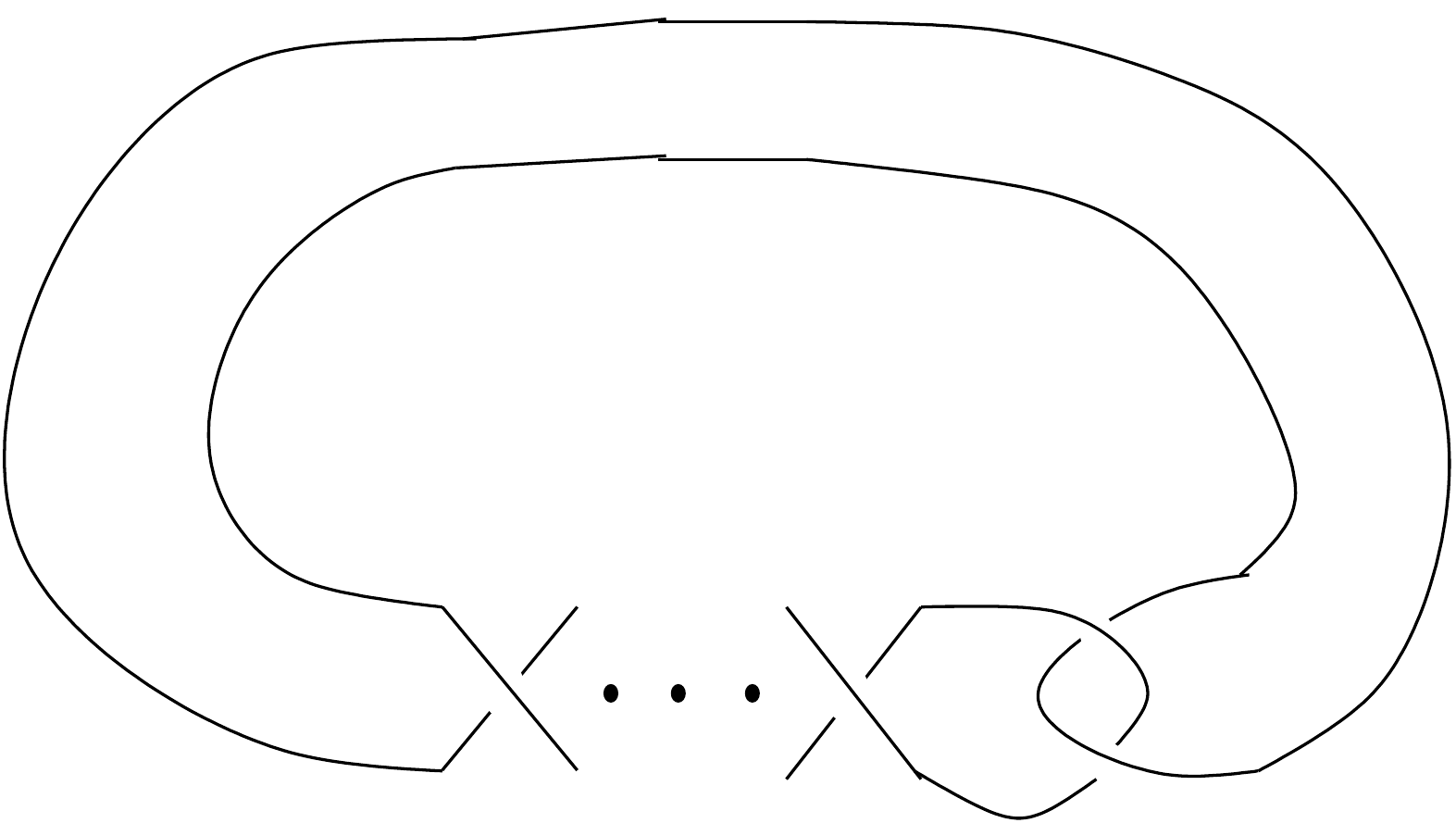}
	\caption{Twist  Knot $W_{n}$}
	\label{fig:TwistKnot}
	\end{figure}	

	\begin{figure}[h!]
\centering
	\includegraphics[height = 1in]{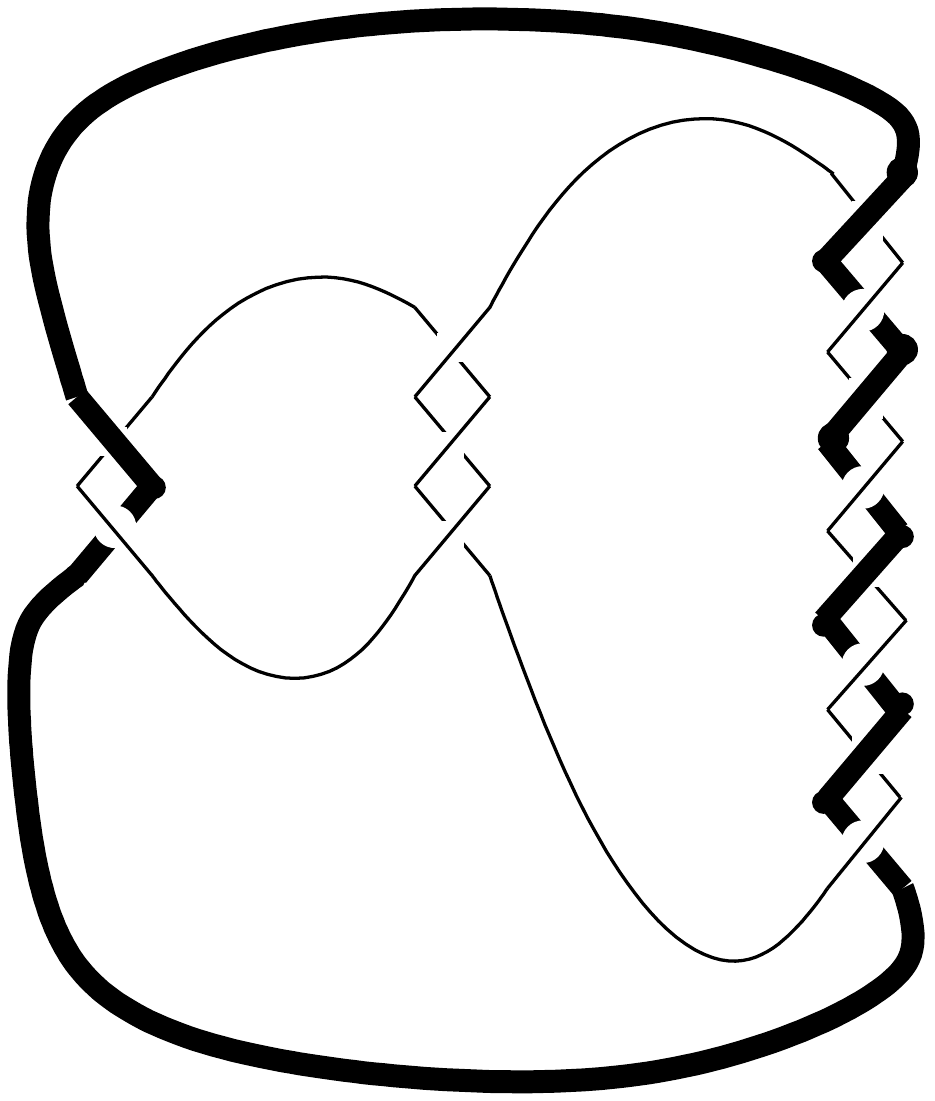}
	\caption{$P$}
	\label{fig:PretzelLink}
	\end{figure}

	\begin{figure}[h!]
\centering
	\includegraphics[height = 3in]{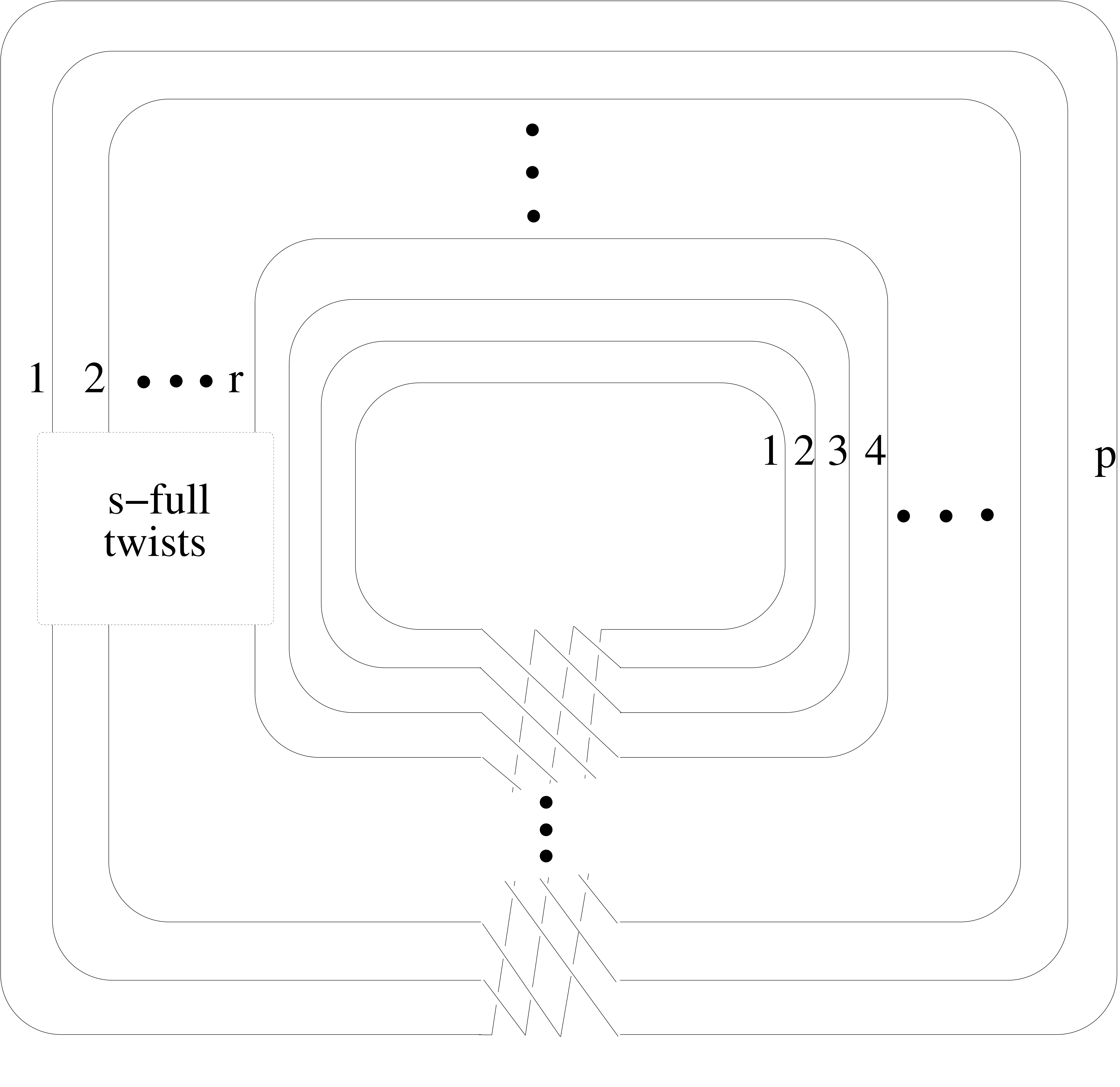}
	\caption{$T(p,q;r,s)$}
	\label{fig:ttk}
	\end{figure}

\subsection{Branched covers}

We assume familiarity with the concept of branched covers; see, for example,~\cite{fox}.
For convenience of the reader we bring the basic definitions and facts here. 
Let $F_{1}$ and $F_{2}$ be closed surfaces.
A map $p:F_{1} \to F_{2}$ is called a {\it branched cover} if it is onto and at every point
$x \in F_{1}$, $p$ is  modeled on the map $D^{2} \to D^{2}$ given by $z \mapsto z^{d}$ 
(for some $d$). with $x$ corresponding to $0$.  Here and throughout this paper 
$D^{2} \subset \mathbb C$ is the unit disk. The number $d$ is called the {\it local degree} at $x$; 
note that the local degrees at different preimages of the same point of $F_{2}$ need not be the same. 
It is easy to see that the set of points of local degree not equal to one is finite.
The image of this set is called the {\it branch set}.

Let $L \subset S^{3}$ be a link.  A map $p:M \to S^{3}$ is a {\it cover branched over $L$} if 
$p|_{p^{-1}(S^{3} \setminus L)}: p^{-1}(S^{3} \setminus L) \to S^{3} \setminus L$ is a cover,
and every $x \in p^{-1}(L)$ has a neighborhood $U$, parametrized as $D^{2} \times (0,1)$, with 
$L \cap U$ corresponding to $\{0\} \times (0,1)$, and $p|_U$
is modeled on the map $D^{2} \times (0,1) \to D^{2} \times (0,1)$ given by
$(z,t) \mapsto (z^{d},t)$.  $d$ is called the {\it local degree} at $x$; note that the local degrees at different
preimages of the same point of $L$ need not be the same.  The {\it degree} of $p$ is defined to be the 
degree of  $p|_{p^{-1}(S^{3} \setminus L)}$.

It is well known (see, for example, Fox~\cite{fox}) that if $M$ is a cover of $S^{3}$ 
branched over $L$ and of degree $d$, then 
$M$ is determined by a representation of $\pi_{1}(S^{3} \setminus L)$ into $S_{d}$, the
group of permutation on $d$ elements.  Finite generation of $\pi_{1}(S^{3} \setminus L)$
and finiteness of $S_{d}$ imply the following well known fact:

\begin{lem}
\label{lem:FiniteleManyCovers}

For given link $L\subset S^{3}$ and integer $d$, there are only finitely many manifolds $M$
that are covers of $S^{3}$, branched over $L$ and of degree at most $d$.
\end{lem}

\subsection{Seifert manifolds}

We assume the reader is familiar with Seifert manifold, that is, circle bundles over orbifolds.   
An orientable Seifert manifold with a given fibration is determined by its {\it Seifert symbols} as follows: 
$(Oy,g;\frac{\beta_1}{\alpha_1},\dots,\frac{\beta_r}{\alpha_r})$ (where $y = o$ or $y = n$
and all other letters represent integers) is the orientable Seifert manifold over
the orientable surface of genus $g \geq 0$ (if $y=o$) or non-orientable surface with $g \geq 1$ cross caps (if $y=n$); each 
fraction $\frac{\beta_{i}}{\alpha_{i}}$ represents a fiber with multiplicity $\alpha_{i}$
(so if $\alpha_{i} = \pm 1$ the fiber is regular and it is exceptional otherwise; all exceptional fibers
must be listed).  The same Seifert manifold can be represented using Seifert symbols 
in infinitely many ways and we refer the reader to Seifert's original
paper~\cite{seifert}  for details.  In particular, we may reorder the exceptional fibers.

\subsection{Montesinos links}

A Montesinos link is a link $L \subset S^3$ that has the form given in Figure~\ref{fig:montesinos}.
	\begin{figure}[h!]
\centering
	\includegraphics[height = 2in]{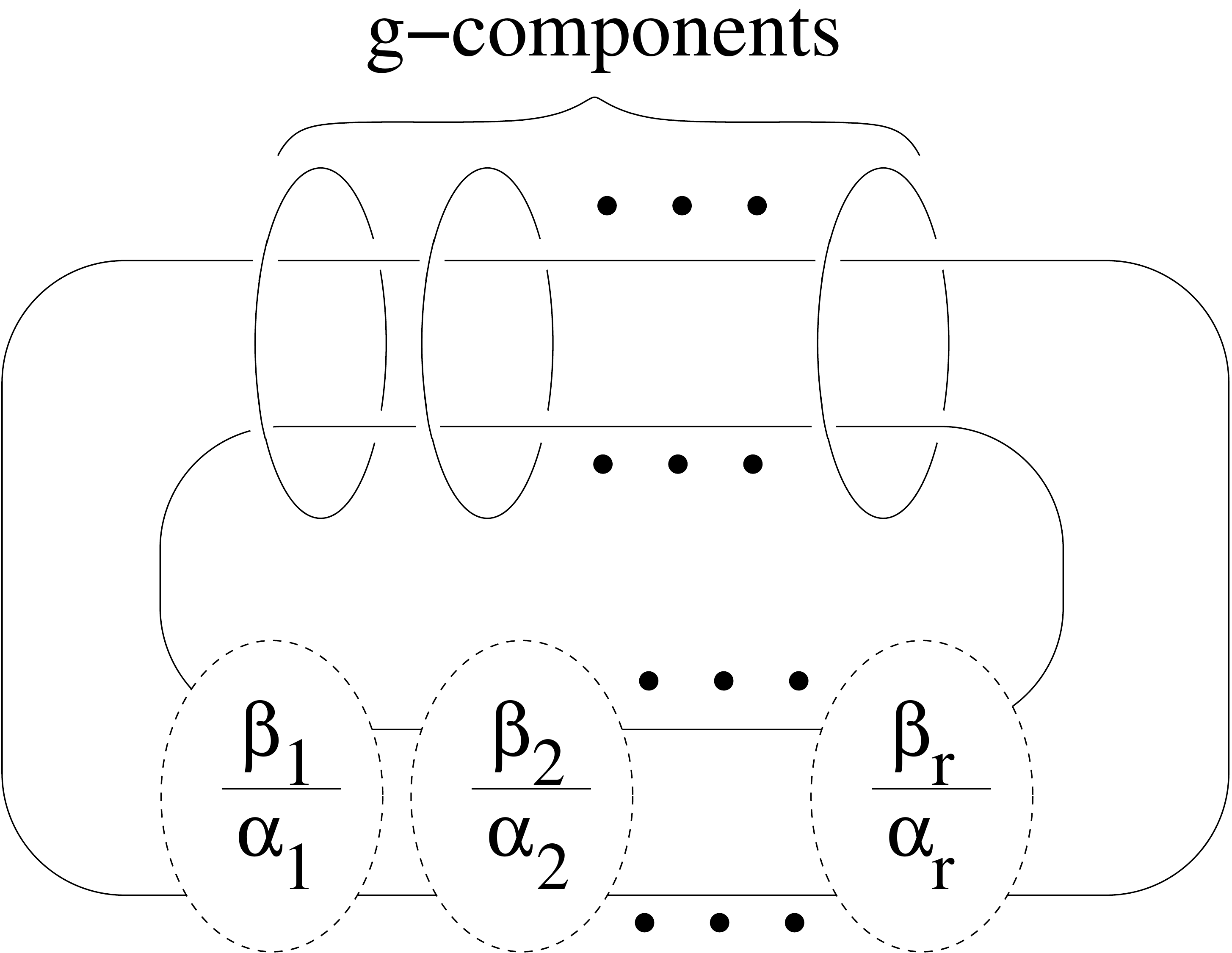}
	\caption{A Montesinos link}
	\label{fig:montesinos}
	\end{figure}
In that figure, $\frac{\beta_i}{\alpha_i}$ represents the rational tangle of that slope.  
We refer the Spanish speaking reader to the original paper of Montesinos~\cite{montesinos},
or~\cite[Section~4.7]{montesinosbook}.  There Montesinos proves the following:

\begin{lem}
\label{lem:montesinos}
The double cover of $S^3$ branched over the Montesinos link $L$ is a Seifert manifold with the following Seifert symbols:
  \begin{enumerate}
   \item When $g=0$: $(Oo,0;\frac{\beta_1}{\alpha_1},\dots,\frac{\beta_r}{\alpha_r})$.
   \item When $g>0$: $(On,g;\frac{\beta_1}{\alpha_1},\dots,\frac{\beta_r}{\alpha_r})$.
  \end{enumerate}
\end{lem}

\subsection{$M_{n}$}

The manifolds studied in this paper, denoted $M_{n}$, $n \in \mathbb Z$, are defined in the introduction 
as the double covers of $S^3$ branched over the links $L_n$.  As seen in Figure~\ref{fig:Ln},
$L_n$ can be seen as a Montesinos link in two distinct ways, giving rise to two 
Seifert fibrations on $M_n$, one with the symbols $(Oo,0;\frac{1}{2},-\frac{1}{2},\frac{-2}{4n-1})$
and the other with the symbols $(On,1;\frac{4n-1}{2})$.

In~\cite[Theorem~2, Pages 111--112]{OrlikBook}, $M_{n}$ is shown belong to a class of manifolds
called  {\it prism manifolds}. 
There it was shown that $M_n$ has {\it exactly} two Seifert fibrations,
and that the fundamental group of $M_{n}$ is finite.
Therefore $M_n $ is a-toroidal.

	\begin{figure}[h!]
\centering
	\includegraphics[height = 5in]{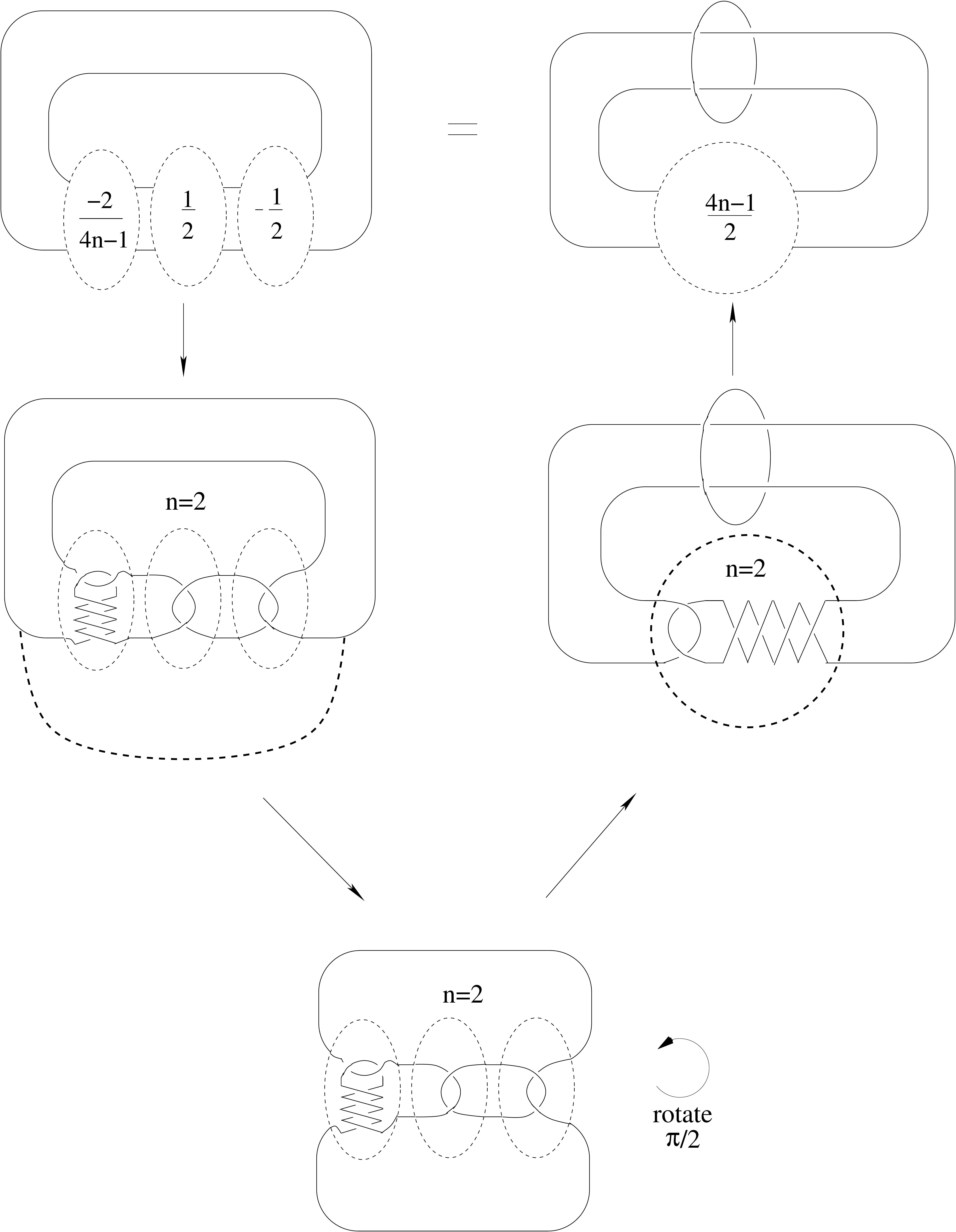}
	\caption{$L_{n}$}
	\label{fig:Ln}
	\end{figure}

We summarize the necessary information about $M_{n}$:

\begin{lem}
\label{lem:Mn}
For each $n \in \mathbb Z$, $M_{n}$ has exactly two Seifert fibrations, one 
with the Seifert symbols $(Oo,0;\frac{1}{2},-\frac{1}{2},\frac{-2}{4n-1})$
and the other with the symbols $(On,1;\frac{4n-1}{2})$.  $M_{n}$ is a-toroidal.
The manifolds $M_{n}$ are distinct, that is, for $n \neq n'$, $M_{n} \not\cong M_{n'}$.
\end{lem}

\subsection{Essential surfaces}

A surface $F$ in a Seifert manifold $M$ is called {\it horizontal} if at every point $p \in F$,
$F$ is transverse to the fiber.   Let $B$ be the base orbifold of $M$.  If $F$ is horizontal, 
then the  
projection $M \to B$ induces a branched cover $F \to B$.  The local degree at $p \in F$
is exactly the multiplicity of the Seifert fiber through $p$.
In particular, if the degree of the cover $F \to B$ is $d$, then 
		$$\chi(F) = d \chi^{orb}(B),$$
where $\chi$ is the Euler characteristic and 	$\chi^{orb}$ is the orbifold Euler characteristic
(see, for example,~\cite{scott}).
$F$ is called {\it vertical} if at every $p \in F$, $F$ is tangent to the fibers;
equivalently, $F$ is the preimage of a 1-manifold embedded in $B$.
See, for example, Jaco~\cite[VI.34]{jaco} for a proof of the following well known fact:
\begin{lem}
\label{lem:EssentialInSeifert}
If $F$ is an essential surface in a Seifert manifold $M$ with base orbifold $B$, 
then one of the following holds:
	\begin{enumerate}
	\item $F$ can be isotoped to be vertical, and $\chi(F) = 0$.
	\item $F$ can be isotoped to be horizontal, and for some integer $d \geq 0$,
	$\chi(F) = d \chi^{orb}(B)$.
	\end{enumerate} 
\end{lem}

\subsection{Hyperbolic manifolds}
\label{subsec:hyperbolic}

By {\it hyperbolic manifold} $M$ we mean a complete finite volume Riemannian 3-manifold locally 
modeled on hyperbolic 3-space $\mathbb H^{3}$.   It is well known that $M$
is the interior of a compact manifold with boundary tori.  The Riemannian metric on $M$ is unique by
Mostow's Rigidity, and induces a volume form on $M$.  Integrating this form we obtain the {\it volume}
of $M$, denoted by $\vol[M]$.  A link $L \subset S^{3}$ is called {\it hyperbolic} if $S^{3} \setminus L$
is a hyperbolic manifold.  By the work of J\o rgensen and Thurston~\cite[Chapter~5]{thurston} (see, 
for example,~\cite{JT} for a detailed description) the set of hyperbolic volumes is well ordered.
Hence in any subset of hyperbolic manifolds there is a (not necessarily unique) manifold of least volume.  Cao and
Meyerhoff~\cite{CaoMeyerhoff} showed that the complement  of figure eight knot and its sister
are the smallest volume hyperbolic manifold with one cusp
(we note that the sister is not the exterior of a knot in $S^3$ and hence will play no role in this work).  
Agol~\cite{agol} showed that the exteriors of the Whitehead link
and the $-2,3,8$ pretzel link are the hyperbolic manifolds of least volume among all hyperbolic 
manifolds with at least two cusps.
Recall that we denoted these links by $W$ and $P$, respectively.

\begin{notation}
\label{notation:v0}
We denote $\vol[S^{3} \setminus W]$ by $\vvv$.
\end{notation}

\subsection{Bundles over $S^1$}
\label{subsec:bundles}
A manifold is a {\it bundle over $S^1$} if it has the form 
$(F \times [0,1])/((x,0) \sim (\phi(x),1))$ where $F$ is a surface and
$\phi:F \to F$ a diffeomorphism; $\phi$ is called the {\it monodromy}
of the manifold and the manifold is denoted by $M_\phi$.
In this subsection we briefly review the structure of 3-manifolds that fiber
over $S^1$, due to Thurston.  We refer to the image of $F \times \{0\}$
as $F \subset M_\phi$.  Then $F$ is essential in $M_\phi$.

Assume that $\chi(F) < 0$.   
By the Thurston--Nielsen classification of surface homomorphisms, 
$\phi$ has one of the following forms:
\begin{enumerate}
\item Pseudo Anosov.
\item Periodic.
\item Reducible.
\end{enumerate}
We refer the reader to \cite{CassonBleiler} for details.  Then $M_\phi$
has one of the following forms:
\begin{enumerate}
\item When $\phi$ is pseudo Anosov, $M_\phi$ is hyperbolic (as described in
Subsection \ref{subsec:hyperbolic}).
\item When $\phi$ is periodic, $M_\phi$ is a Seifert manifold and $F$ is horizontal
(by construction).  
\item When $\phi$ is reducible, there is a collection of disjointly embedded essential circles 
$C \subset F$ so that $\phi(C) = C$.  Then the image of $C \times [0,1]$
in $M_\phi$ is a collection of disjointly embedded essential tori, say $T$.  
Denote the closures of the component of $M_\phi$ cut open along $T$ by $V_i$
($i=1,\dots,n$).  Then every $V_i$ fibers over $S^1$ with fiber 
$F \cap V_i$, and  for an appropriate choice of $C$, the monodromy $\phi|_{F \cap V_i}$ is either
pseudo Anosov or periodic; accordingly, $V_i$ is either hyperbolic or a Seifert manifold.
When dealing with reducible monodromy, we will always assume that $C$ was chosen so that 
$T$ and $V_{i}$ have these properties.

\end{enumerate}

\section{The knots $P_{n}$}
\label{sec:Pn}

Recall the definition of a twisted torus knot from Subsection~\ref{subsection:notation}.
In this section we study the knots $P_{n}$.  Recall that $P_{n}$ was obtained from $P$
(the $-2,3,8$ pretzel link)
by $1/n$ Dehn surgery on the unknotted component.  The trace of the unknotted component is
an unknot in $S^{3}$.

\begin{prop}
\label{prop:Pn}
The knot $P_{n}$ is the $T(5, 5n+1;2,1)$ twisted torus knot.  Moreover, $T(5,5n+1;2,1)$
naturally embeds in an unknotted solid torus $V$ and the trace of the unknotted component
of $P$ is a core of the complementary solid torus.
\end{prop}

\begin{proof}
For $n=0$, see Figure~\ref{fig:TwistedTorusKnot}; the trace of the unknotted component is in boldface.
For all other $n$, perform an $n$ Dehn twist on the unknotted component of $P$ as given in the top right corner of that figure.
	\begin{figure}[h!]
\centering
	\includegraphics[height = 7in]{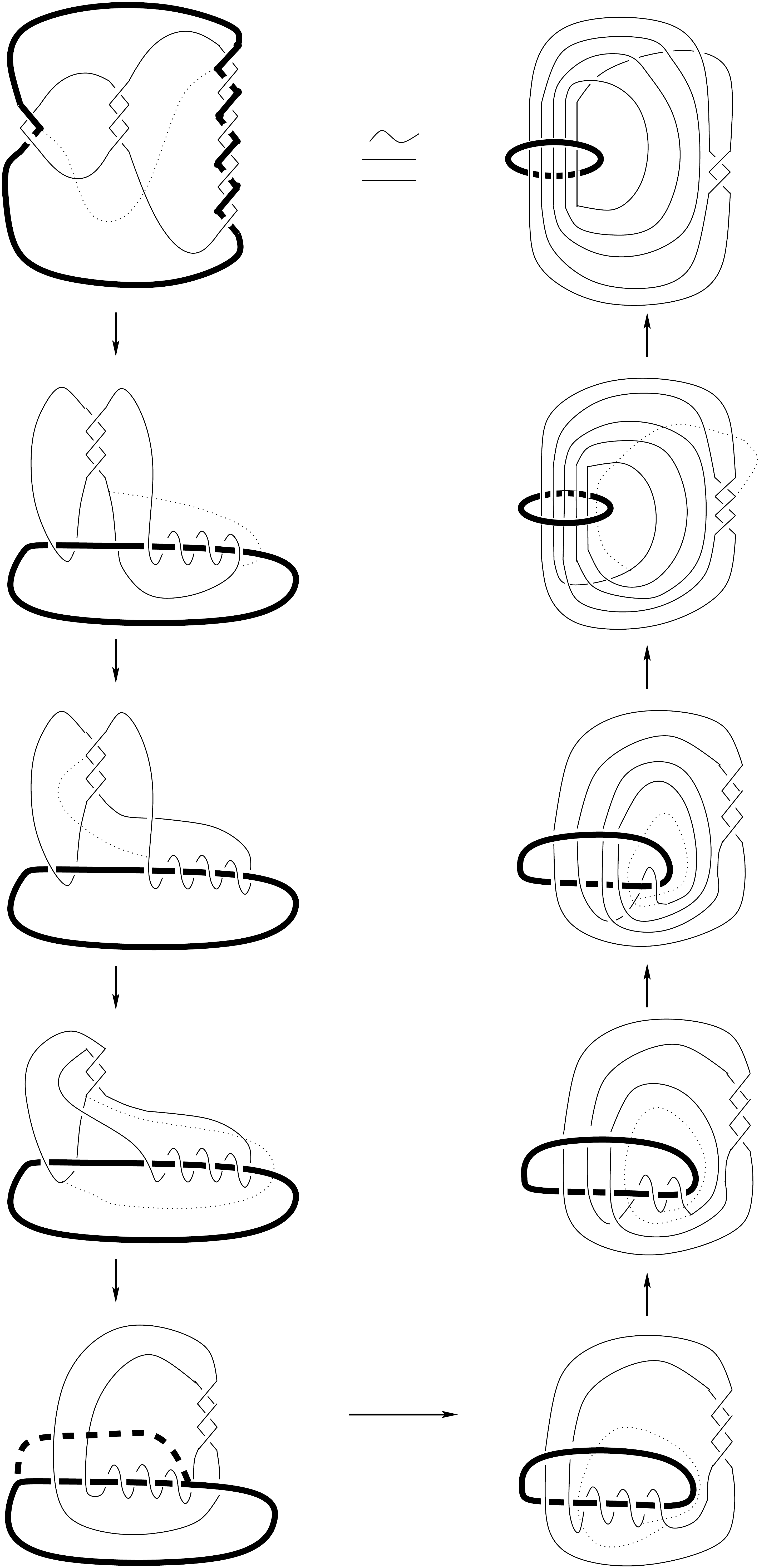}
	\caption{$P_{n}$ is a Twisted Torus Knot}
	\label{fig:TwistedTorusKnot}
	\end{figure}
\end{proof}

\section{Knots of small volume}
\label{sec:SmallVolumeKnots}

\begin{prop}
\label{prop:SmallVolumeKnots}
Let the knots $W_{n}$ and $P_{n}$ be as in Subsection~\ref{subsection:notation}, and let
$\vvv$ be as in Notation~\ref{notation:v0}.
The set 
$$\mathcal{L} = \{ L \subset S^{3} | L \mbox{ \rm is a hyperbolic link and } \vol[S^{3} \setminus L] < \vvv \}$$ 
consists entirely of knots, and all but 
finitely many of these knots are of the form $W_{n}$ or $P_{n}$.
\end{prop}

\begin{proof}
By Agol \cite{agol}, the minimal volume hyperbolic manifolds with two cusps are $S^{3} \setminus W$ and $S^{3} \setminus P$,
and $\vol[S^{3} \setminus W] = \vol[S^{3} \setminus P] = \vvv$.  
Hence every link in 
$\mathcal{L}$ is a knot. 

Let $\hat{\mathcal{L}}$ be the knots in ${\mathcal{L}}$ that are not of the form $W_{n}$ or $P_{n}$.
Assume, for a contradiction, that $\hat{\mathcal{L}}$ is infinite.
Gordon and Luecke proved that knots are determined by their complements~\cite{GL}, hence
there are infinitely many distinct manifolds in $\{ S^{3} \setminus L | L \in \hat{\mathcal{L}}\}$.


Thus $\{S^3 \setminus L | L \in \hat{\mathcal{L}} \}$ is an infinite
collection of hyperbolic manifolds, all of volume at most $\vvv$. By
J\o rgensen and Thurston,
there exists an infinite sub-collection $\hat{\mathcal{L}}' \subset
\hat{\mathcal{L}}$ and a hyperbolic manifold $X$ so that every manifold
in
$\{S^3 \setminus L | L \in \hat{\mathcal{L}}'\}$ is obtained from $X$ by
Dehn filling, and $\vol[X] \leq \vvv$  (see, for example,
\cite[Theorem~E.4.8]{BenedettiPetronio}).  Since the manifolds in
$\{S^3 \setminus L | L \in \hat{\mathcal{L}}'\}$ have cusps,
$X$ must have at least two cusps.  By Agol, $X \cong S^{3} \setminus
W$ or $X \cong S^{3} \setminus P$.


It is easy to see that if a knot is obtained by filling one component
of $S^{3} \setminus W$ then it is of the form $W_{n}$.  If infinitely many non-trivial knots are obtained by filling a boundary component of $S^{3} \setminus P$ then 
(since the knots are in $S^{3}$) the boundary component filled must correspond to the unknotted component 
of $P$ and the slope filled must be of the form $1/n$
with respect to the usual meridian and longitude; hence the knot is of the form $P_{n}$.  This contradicts our assumption.
\end{proof}

\section{Proof of Theorem~\ref{thm:main}}
\label{sec:proof}

In this section we prove Theorem~\ref{thm:main}.

First we prove:

\begin{lem}
\label{lem:NotDoubleCover}
For all but finitely many $n$, $M_{n}$ is not the double cover of $S^{3}$ branched over 
a hyperbolic link $L$ with $\vol[S^{3} \setminus L] < \vvv$.
\end{lem}

\begin{proof}
Note that for a given link $L \subset S^{3}$, the double cover of $S^{3}$ branched along $L$ 
is unique.  Hence, it suffices to show that there are finitely many links in $S^{3}$ with volume less than $\vvv$ whose 
double cover is $M_{n}$, for some $n$.  By Proposition~\ref{prop:SmallVolumeKnots} we only need to consider knots of the form $W_{n}$
and $P_{n}$.

Note that $W_{n}$ are twist knots, see Figure~\ref{fig:TwistKnot}.   These are very simple Montesinos links, and 
the double covers of $S^{3}$ branched along them are lens spaces, and hence
are not $M_{n}$.  Thus it suffices to show that only finitely many $M_{n}$'s are double covers of $S^{3}$ branched along
$P_{n}$.

Let $U$ be the unknotted component of the $-2,3,8$ pretzel link.  Let $V = S^{3} - N(U)$, and let $k \subset V$
be the knotted component of the $-2,3,8$ pretzel link, that is, $k$ is the
image of the twisted torus knot $T(5, 1; 2, 1)$ in $V$;
see the top right knot in Figure~\ref{fig:TwistedTorusKnot}.  
Note that $(S^{3}, P_{n})$ is obtained from $(V,k)$ by Dehn filling.  Let $\wt V$ be the double cover of $V$
branched along $k$.  Then the double cover of $S^{3}$ branched along $P_{n}$ is obtained from $\wt V$
by Dehn filling.  

It is easy to see from Figure~\ref{fig:TwistedTorusKnot} that  we may
choose a parametrization of $V$ as $D^2 \times S^1$ so that
 $k$ intersects every disk of the form $D^{2} \times \{pt\}$
transversally at exactly five points.   Hence $\wt V$ inherits a structure of a fiber bundle over $S^{1}$ with fiber the
double cover of $D^{2}$ branched over five points; it is easy to see that this fiber is the surface of genus 2 and one puncture,
denoted by $\f21$.

By Thurston's classification of bundles over $S^{1}$, the monodromy of $\wt V$ has one of the following forms
(recall Subsection~\ref{subsec:bundles}):
	\begin{enumerate}
	\item Pseudo Anosov.
	\item Periodic.
	\item Reducible. 
	\end{enumerate}

\bigskip

\noindent
If the monodromy is pseudo Anosov, then $\wt V$ is hyperbolic.  Then by Thurston's Dehn Surgery Theorem,
all but finitely many fillings are hyperbolic, and hence are not $M_{n}$.  We assume from now on that the monodromy is
not pseudo-Anosov.

\bigskip

\noindent
If the monodromy is periodic, then $\wt V$ is a Seifert manifold.  Then the Seifert fibration of $\wt V$ induces a Seifert
fibration on every manifold obtained from it by Dehn filling but one, and the core of the attached solid torus is a fiber.  
Assume that one of these manifolds is $M_{n}$.  Then the Seifert fibration
of $\wt V$ is obtained from a Seifert fibration of $M_{n}$ by removing a fiber.  By Lemma~\ref{lem:Mn}, 
$M_{n}$ admits exactly two Seifert fibrations, one with Seifert symbol $(On,1;\frac{4n-1}{2})$,
and the other with Seifert symbol $(Oo,0;\frac{1}{2},-\frac{1}{2},\frac{-2}{4n-1})$.
Note that the fiber removed may be exceptional or regular.  The possible
base orbifolds and their orbifold Euler characteristics after removing a fiber are: 
	\begin{enumerate}
	\item A M\"obius band with no singular points; $\chi^{orb}(\mbox{base})=0$.
	\item A M\"obius band with one singular point of index $2$; $\chi^{orb}(\mbox{base})=- \frac{1}{2}$.
	\item A disk with exactly three singular points, two of index 2 and one of index $4n-1$;
	$\chi^{orb}(\mbox{base})=-1 +{ \frac{1}{4n-1}}$.
	\item A disk with exactly two singular points, both of index 2;	$\chi^{orb}(\mbox{base})=0$.
	\item A disk with exactly two singular points, one of index 2 and one of index
	$4n-1$; $\chi^{orb}(\mbox{base})= -\frac{1}{2} + { \frac{1}{4n-1}}$.
	\end{enumerate}

Since $\f21$ is a fiber in a fibration over $S^{1}$, it is essential in $\wt V$.
By Lemma~\ref{lem:EssentialInSeifert},
after isotopy we may assume it is horizontal or vertical, but the latter is impossible
since $\chi(\f21) = -3 \neq 0$.  Hence we may assume that $\f21$ is horizontal and branch covers
the base orbifold, where the indices of the singular points are the degrees of the branching.
Applying Lemma~\ref{lem:EssentialInSeifert} again we see that the Euler 
characteristic of $\f21$ equals the degree of the cover (say $d$) times the orbifold Euler characteristic of
the base.  In Cases~(1) and ~(4) this is impossible.  

In Case~(3), the equation $\chi(\f21) = d \chi^{orb}(\mbox{base})$ means that for some $d$,
$-3 = d(-1 + \frac{1}{4n-1})$.  Solving for $n$ we see that $n = \frac{2d-3}{4d-12}$.  It is easy to see
that only finitely many integral values of $d$ correspond to integral values of $n$; hence only finitely many
$M_{n}$ can be obtained in this way, that is to say, for only finitely many values of $n$, there exists $m$,
so that $M_{n}$ is the double cover of $S^{3}$ branched over $P_{m}$.

In Case~(5), the equation $\chi(\f21) = d \chi^{orb}(\mbox{base})$ means that for some $d$,
$-3 = d(-\frac{1}{2} + \frac{1}{4n-1})$.  Solving for $n$ we see that $n = \frac{3d-6}{4d-24}$.  In this case too, it is easy to see
that only finitely many integral values of $d$ correspond to integral values of $n$; hence only finitely many
$M_{n}$ can be obtained in this way.

In Case~(2),  $\f21$ is the cover of a M\"obius band branched over one singular point of index 2.  Since
$\f21$ is orientable, the cover factors through the orientation double cover, which is an annulus with two singular points,
each of index 2.  Since the indices are 2, $\f21$ is a cover of the annulus of an 
even degree, say $2m$.
Hence $d = 4m$, and the equation $\chi(\f21) = d \chi^{orb}(\mbox{base})$ 
becomes: $-3 = 4 m \chi^{orb}(\mbox{base}) = 4 m \frac{-1}{2} =
-2m$.  This is impossible, since $m$ is an integer.

\bigskip

\noindent Finally, we consider reducible monodromy. 
In that case, the monodromy induces a torus decomposition on $\wt V$.
We denote the components of this decomposition by
$V_{1},\dots,V_{k}$.  Recall from Subsection~\ref{subsec:bundles} that for each $i$, $V_{i}$ is either a hyperbolic manifold or a Seifert manifold, 
with boundary a non-empty collection of tori that are incompressible in $V_{i}$ and in $\wt V$.
We may assume that the boundary of $\wt V$ is contained in 
$V_{1}$.  Hence $V_{1}$ has at least two boundary components.

There is a natural correspondence between Dehn fillings of $\wt V$ and Dehn fillings of $V_1$ along $\partial\wt V\subset\partial V_1$ 
 obtained by filling the same curve $\alpha$.  When filling
$V_{1}$, the components of $\partial V_1 \setminus \partial \wt V$ are not filled.  We denote corresponding fillings
by $\wt V(\alpha)$ and $V_{1}(\alpha)$.

If $V_{1}$ is hyperbolic, then for all but finitely many fillings $V_{1}(\alpha)$ is hyperbolic.  
When $V_{1}(\alpha)$ is hyperbolic, $\partial V_{1}(\alpha) = \partial V_{1} \setminus \partial \wt V$ is a 
non-empty collection of essential tori
in $V_{1}(\alpha)$.  Since these tori are also incompressible into $\wt V \setminus \mbox{int}(V_{1})$,
they are essential in $\wt V(\alpha)$; hence $\wt V(\alpha)$ is toroidal.
By Lemma~\ref{lem:Mn}, $M_{n}$ is a-toroidal.  Hence, only finitely many fillings of $\wt V$ give a manifold of the 
form $M_{n}$ in this case.

If $V_{1}$ is a Seifert manifold then for all but one slope $\alpha$, $V_1(\alpha)$ is a Seifert manifold;
recall that $\partial V_1(\alpha) \neq \emptyset$.
If $\partial V_{1}(\alpha)$ is incompressible in $V_{1}(\alpha)$,
then as above, $\wt V(\alpha)$ is toroidal and hence is not $M_{n}$.  The only Seifert manifold
that has a compressible boundary is the solid torus.  The only Seifert manifolds that can be filled
to give a solid torus are  solid tori with one fiber removed.  If the fiber is the core
of the solid torus then $V_{1}$ is a torus cross an interval; but by construction 
$\partial V_{1} \setminus \partial \wt V$
is essential in $\wt V$, and in particular it is not boundary parallel.  Hence, the fiber removed is a 
regular fiber in a fibration of the solid torus 
that has exactly one exceptional fiber.  Here is an alternate description: since the solid torus is
a Seifert manifold over the disk with at most one exceptional fiber, removing a fiber yields a Seifert
manifold over the annulus with at most one exceptional fiber.  Essentiality of $\partial
V_{1} \setminus \partial \wt V$ implies that there is an exceptional fiber.

On the other hand, by construction, $V_{1}$ fibers over $S^{1}$ with fiber $\f21 \cap V_{1}$ (recall
Subsection~\ref{subsec:bundles}).  
For convenience we note that since
$\partial \f21$ has exactly one component,
$\f21 \cap V_{1}$ is connected.  Since $V_{1}$ is not $T^{2} \times [0,1]$, it is not a fiber bundle over
$S^{1}$ with fiber an annulus.  Hence, $\f21 \cap V_{1}$ is not an annulus.  Since $\f21 \cap V_{1}$
is a fiber in a fibration over $S^{1}$, it is essential and can be isotoped to be vertical or horizontal;
since it is not an annulus, it is horizontal.
Hence the slope defined by $\f21 \cap \partial \wt V$ and the slope defined by the regular
fiber in the Seifert fibration of $V_{1}$ are distinct.

Since $V_{1}$ is a Seifert manifold over the annulus with one exceptional fiber, $V_{1}(\alpha)$
is a solid torus if and only if $\alpha$ intersects a regular fiber exactly once.
Denoting the regular fiber by $f$, this condition can be written as $\Delta(f,\alpha) = 1$,
where $\Delta(\cdot,\cdot)$ is the absolute value of the algebraic intersection
number on $\partial \wt V$.

On the other hand,  $\partial \f21$ projects to the boundary of the meridian of $V$.
Recall that $V = S^{3} \setminus \mbox{int}(N(U))$, where $U$ denotes the
unknotted component of the $-2,3,8$ pretzel link.  We obtain $P_{n} \subset S^{3}$ from
$k \subset V$ by Dehn filling; since the ambient manifold is $S^{3}$ the meridian
of the attached solid torus intersects the meridian of $V$ at exactly one point, say $x$.
By construction, the preimage of the meridian disk of
$V$ is $\f21$; thus the preimage of the 
meridian of $V$ is $\partial \f21$.  The preimage of the meridian
of the solid torus attached to $V$ is (one or two) meridians of the solid torus attached 
to $\wt V$, defining the slope $\alpha$.
Hence, $\partial \f21 \cap \alpha$ is contained in the preimage of $x$.
The preimage of $x$ under a double cover consists of two points.
We conclude that $\Delta(\partial \f21,\alpha) \leq 2$.

Hence there are two distinct slopes on $\partial \wt V$, denoted $f$
and $\partial \f21$, so that $\alpha$, the slope filled, fulfills the following:
\begin{enumerate}
\item $\Delta(f,\alpha) = 1$, and
\item $\Delta(\partial \f21,\alpha) \leq 2$.
\end{enumerate} 

It is an easy exercise to show that there are only finitely many slopes that fulfill these
two conditions simultaneously.

These completes the proof of Lemma~\ref{lem:NotDoubleCover}.
\end{proof}

We now prove Theorem~\ref{thm:main}.  Since $M_{n}$ is the double cover of $S^{3}$ branched
over $L_{n}$ and $\vol[S^{3} \setminus L_{n}] = \vvv$, $\lv[M_{n}] \leq 2\vvv$.
$\lv[M_{n}] < 2\vvv$ if and only if one of the following holds:
	\begin{enumerate}
	\item $M_{n}$ is a double cover of $S^{3}$ branched over a hyperbolic link $L$ with
	$\vol[S^{3} \setminus L] < \vvv$.
	\item $M_{n}$ is the $p$-fold cover of $S^{3}$, branched over a hyperbolic link $L$ with
	$\vol[S^{3} \setminus L] < 2\vvv/p$, for $p \geq 3$.
	\end{enumerate}
	
Lemma~\ref{lem:NotDoubleCover} shows that there are only finitely many manifolds $M_{n}$ in Case ~(1).

In Case~(2), we only consider links $L$ with $\vol[S^{3} \setminus L] < 2\vvv/p \leq 2\vvv/3$.
We claim that this set is finite.  To see that, assume it is infinite.  Recall again  that hyperbolic links with $\vol[S^3\setminus L]\leq 2\vvv/3<\vvv$ are knots, as mentioned in Proposition~\ref{prop:SmallVolumeKnots}. Hence, by Gordon and Luecke~\cite{GL} there are infinitely many $S^3\setminus L$ with $\vol[S^3\setminus L]\leq 2\vvv/3.$ Then, as in the proof of Proposition~\ref{prop:SmallVolumeKnots}, we have an infinite sub-collection of the set $\{S^3\setminus L:\textrm{ L is hyperbolic and }\vol[S^{3} \setminus L] \leq 2\vvv/3\}$ and every manifold in this sub-collection is obtained by filling a hyperbolic manifold with at least two cusps and volume at most $2\vvv/3.$
 But Agol~\cite{agol} shows that any hyperbolic
manifold with at least two cusps has volume at least $\vvv$, contradicting our assumption there were infinitely many such links.   Fix a link $L$ with  $\vol[S^{3} \setminus L] < 2\vvv/p \leq 2\vvv/3$.
To obtain manifolds $M$ with $\lv[M] < 2\vvv$ as a cover branched over $L$, we consider
covers of degree $p < 2\vvv/\vol[S^{3} \setminus L]$.    (We note that by Cao and
Meyerhoff~\cite{CaoMeyerhoff} $\vol[S^{3} \setminus L] \geq 2$, and hence $p = 3$.)  
By Lemma~\ref{lem:FiniteleManyCovers},
there are only finitely many manifolds that cover $S^{3}$
branched over $L$ and of bounded degree, and we conclude that in Case~(2) there
are only finitely many manifolds.

This completes the proof of Theorem~\ref{thm:main}.

\bibliographystyle{plain} %
\bibliography{JairYoav} %
\end{document}